\documentclass[reqno, 11pt]{amsart}
\usepackage{amssymb,amsmath,amsfonts}
\usepackage[margin=1in]{geometry}
\usepackage{dsfont}
\usepackage{color}
\usepackage{graphicx}
\usepackage{latexsym}
\usepackage{amsmath}
\usepackage{amssymb}
\usepackage{graphics}
\usepackage[dvips]{epsfig}
\usepackage{mathrsfs}
\usepackage{mathtools}
\usepackage{leqno}
\usepackage{stmaryrd,cite}

\usepackage{cases}
\usepackage{enumerate}
\usepackage[usenames,dvipsnames,svgnames]{xcolor}
\definecolor{refkey}{rgb}{1,0,0.5}
\definecolor{labelkey}{rgb}{0,0.4,1}
\usepackage{todonotes}
\usepackage{marginnote}

\presetkeys{todonotes}{fancyline, color=green!40}{}

\makeatletter
\renewcommand{\@todonotes@drawMarginNoteWithLine}{%
	\begin{tikzpicture}[remember picture, overlay, baseline=-0.75ex]%
	\node [coordinate] (inText) {};%
	\end{tikzpicture}%
	\marginnote[{
		\@todonotes@drawMarginNote%
		\@todonotes@drawLineToLeftMargin%
	}]{
		\@todonotes@drawMarginNote%
		\@todonotes@drawLineToRightMargin%
	}%
}
\makeatother
\usepackage{listings}
\usepackage{ifpdf}
\ifpdf
\usepackage[CJKbookmarks=true,
         hyperindex=true,
         pdfstartview=FitH,
         bookmarksnumbered=true,
         bookmarksopen=true,
         colorlinks=true,
         citecolor=blue,
         linkcolor=blue,
         urlcolor=blue,
         pdfborder=001,
         pdfauthor={},
         pdftitle={},
         pdfkeywords={},
         ]{hyperref}
\else
\usepackage[
         hypertex,
         hyperindex,
         linkcolor=blue,
         unicode,
         citecolor=blue%
         ]{hyperref}
\fi
\usepackage{cleveref}

\allowdisplaybreaks

\numberwithin{equation}{section}
\newtheorem{theorem}{Theorem}[section]

\newtheorem{lemma}[theorem]{Lemma}
\newtheorem{proposition}[theorem]{Proposition}
\newtheorem{remark}[theorem]{Remark}

\newcommand{\be}{\begin{equation}}
\newcommand{\ee}{\end{equation}}

\newcommand{\bee}{\begin{equation*}}
\newcommand{\eee}{\end{equation*}}

\newcommand{\bse}{\begin{subequations}}
\newcommand{\ese}{\end{subequations}}

\newcommand{\bs}{\begin{split}}
\newcommand{\es}{\end{split}}

\begin{document}

\author{Hairong Liu$^{1}$}\thanks{$^{1}$School of Mathematics and Statistics, Nanjing University of Science and Technology,  Nanjing 210094 P.R.China.
E-mail : hrliu@njust.edu.cn}

\author{Xiaoping Yang$^{2}$}\thanks{$^{2}$School of Mathematics,  Nanjing University, Nanjing 210093, P.R. China.
E-mail: xpyang@nju.edu.cn}

\title[] {Strong unique continuation property for fourth order Baouendi-Grushin type subelliptic operators
with strongly singular potential}

\begin{abstract}
In this paper, we prove the strong unique continuation property for the following fourth order degenerate elliptic equation
\begin{align*}
\Delta^2_{X}u=Vu,
\end{align*}
where $\Delta_{X}=\Delta_{x}+|x|^{2\alpha}\Delta_{y}$ ($0<\alpha\leq1$), with $x\in\mathbb{R}^{m}, y\in\mathbb{R}^{n}$,
denotes the Baouendi-Grushin type subelliptic operators, and
the potential $V$ satisfies the strongly singular growth assumption
$|V|\leq \frac{c_0}{\rho^4}$, where
\begin{align*}
\rho=\left(|x|^{2(\alpha+1)}+(\alpha+1)^2|y|^2\right)^{\frac{1}{2(\alpha+1)}}
\end{align*}
 is the gauge norm.
The main argument is to introduce an Almgren's type frequency function for the solutions, and show its monotonicity to obtain a doubling estimate based on setting up some refined Hardy-Rellich type inequalities
on the gauge balls with  boundary terms.


\end{abstract}

\maketitle

\section{Introduction and main results}

This paper is devoted to studying the strong unique continuation property for  fourth order
Baouendi-Grushin type subelliptic operators with strongly singular potential.
A differential operator $L$ is
said to have the strong unique continuation property in $\Omega$ if the only solution of $Lu=0$
which vanishes of infinite order at a point $x_0\in\Omega$ is $u\equiv0$.  In the past decades the strong unique continuation property for solutions to various kinds of
 partial differential equations has attracted a large number of researchers and induced
 many interesting and intensive results. The results for the second order operator
\begin{align}\label{sch}
-\Delta u=Vu
\end{align}
go back to the work of Carleman \cite{c1939}, who solved the uniqueness problem in $\mathbb{R}^{2}$ with  bounded potentials.  Cordes \cite{c1956} and Aronszajn \cite{a1957} extended the strong unique continuation property to second-order equations in $\mathbb{R}^{n}$. Subsequent developments in this direction is Jerison and Kenig's result \cite{jk1985}  on the strong unique continuation property for (\ref{sch}) with $V\in L_{loc}^{n/2}(\mathbb{R}^{n})$. There is a large amount of work on strong unique continuation and quantitative unique continuation  for second order elliptic operators, achieved by different Carleman type estimates (cf. \cite{DZ2019, DLW2021, h2001} and references therein). On the other hand, Garofalo and Lin \cite{gl1,gl2} presented a geometric-variational approach to the strong unique continuation by using the frequency function. Their method is based on establishing the doubling estimate
which in turn depends on the monotonicity property of the frequency function.  It is worth pointing out that
the frequency function was first introduced by Almgren \cite{alm} for harmonic functions.

On the contrary, it turns out that unique continuation property  is generically not true for subelliptic
operators. Bahouri \cite{B1986} showed that unique continuation property is not true for even smooth
and compactly supported perturbations of the sub-Laplacian on the Heisenberg
group. Garofalo and Lanconelli \cite{GL} showed some positive results of the strong unique continuation property to the sub-Laplace
$\Delta_{H}u=Vu$ on the Heisenberg group provided  $u$ has some symmetries. Closely related to  sub-Laplacian on
the Heisenberg group is the following  Baouendi-Grushin operator
\begin{align}\label{Grushin0}
\Delta_{X}=\Delta_{x}+|x|^{2\alpha}\Delta_{y}, \quad \alpha>0,
\end{align}
which is the operator of a sum of squares of
the following vector fields $X=\{X_1,\cdots,X_{m+n}\}$
\begin{equation}\label{vector1}
X_i=\partial_ {x_i},
\hspace{1mm}i=1,\cdots,m,\hspace{2mm}
X_{m+j}=|x|^{\alpha} \partial_{y_j},\hspace{1mm}j=1,\cdots,n, \quad m,n\geq1.
\end{equation}
Here $\alpha>0$ is a fixed parameter, $x=(x_1,\cdots,x_{m})\in
\mathbb{R}^{m}$, $y=(y_1,\cdots,y_{n})\in \mathbb{R}^{n}$ and $\Delta_{x}, \Delta_{y}$ denote the standard Laplacian.
When $\alpha=0$, $\Delta_{X}$ is just the standard Laplacian. For $\alpha>0$, $\Delta_{X}$ is
elliptic for $x\neq 0$ and degenerate on the
submanifold $\{0\}\times \mathbb{R}^{n}$, and  it is not translation invariant in $\mathbb{R}^{N}$.
We mention that when $\alpha=1$, operator $\Delta_{X}$ is
connected to the sub-Laplacians on the Heisenberg-type groups (see e.g., \cite{G2015}).
We recall that a more general class of operators
modelled on $\Delta_{X}$ was first introduced by Baouendi, who studied the Dirichlet problem in weighted
Sobolev spaces in \cite{B1967}. When
$\alpha\in\mathbb{N} $ the operator in (\ref{Grushin0}) was
studied by Grushin \cite{Grushin1970,Grushin1971}, who established
its hypoellipticity. If $\alpha=2n$, with $n\in \mathbb{N}$, then
$\Delta_{X}$ is a sum of squares of $C^{\infty}$ vector
field satisfying H\"{o}rmander finite rank.
Franchi and Lanconelli \cite{FL} studied embedding theorems for Sobolev spaces related to
general vector fields (\ref{vector1}).  For some other
interesting properties related to the Baouendi-Grushin operators see \cite{grushin2020,chen2017,liu}.

The  strong unique continuation  of the following  Baouendi-Grushin operators
\begin{align}\label{second-order}
-\Delta_{X}u=Vu
\end{align}
was  established by Garofalo \cite{G1993} by  introducing an Almgren-type frequency function,  under the assumptions
\begin{equation}\label{V2}
|V|\leq \frac{f(\rho)}{\rho^2}\psi,
\end{equation}
for some non-decreasing $f:(0,R_0)\rightarrow \mathbb{R}^{+}$  such that
\begin{equation}\label{f}
\int_0^{R_0}\frac{f(r)}{r}dr< \infty,  \quad  \mbox{for some }\quad R_0>0.
\end{equation}
Here $\rho$ means a gauge norm associated to the vector field  (\ref{vector1}), and $0\leq\psi\leq1$ is a weight function (see (\ref{gauge}) and (\ref{psi}) below for details).
In \cite{G1993}, for the case $V=V^{+}-V^{-}$ satisfies
\begin{align}\label{V3}
0<V^{+}\leq C\frac{\psi}{\rho^2},\quad 0\leq V^{-}\leq\delta\frac{\psi}{\rho^2},
\end{align}
for some $\delta>0$ small enough,  the author established the  unique continuation property in the sense of solution decays exponentially at one point (see (\ref{exp}) below). Note that the weight $\psi$  degenerates
on the submanifold $\{x=0\}$ and so the result in \cite{G1993} does not allow to take $V\in L^{\infty}$.
Recently,  Banerjee, Mallick \cite{BM2020} obtained  the strong unique continuation property of (\ref{second-order}) in the special case $\alpha=1, n=1$ under the weaker assumptions
\begin{equation}\label{V4}
|V|\leq \frac{f(\rho)}{\rho^2}, \quad \mbox{or}\quad |V|\leq C\frac{\psi^{\varepsilon}}{\rho^2}, \quad \mbox{for some}\quad  \varepsilon>0,
\end{equation}
where $f$ is as in (\ref{f}).

In this paper, we consider the following  fourth order
Baouendi-Grushin type subelliptic  equation
\begin{equation}\label{L}
Lu\equiv -\Delta^2_{X}u+Vu=0,
\end{equation}
with $0<\alpha\leq 1$, where the potential $V$ satisfies the following strongly singular,
\begin{equation}\label{V}
|V|\leq \frac{c_0}{\rho^4},
\end{equation}
for some $c_0$ small  depending only on $m, n$ and $\alpha$.

Even for higher order uniform elliptic operators, the strong unique continuation property is quite different with and
more complex than second order operators.
Alinhac \cite{a} constructed a strong uniqueness counterexample for differential operators $P$ of any order in $\mathbb{R}^2$, under the condition that $P$ has two simple, nonconjugate complex characteristics. Colombini and Grammatico \cite{cg1999} gave strong uniqueness for  a differential inequality related to bi-Laplacian
\begin{align}\label{his1}
|\Delta^{2}u(x)|\leq C_1\frac{|u|}{|x|^4}+C_2\frac{|\nabla u|}{|x|^3}+C_3\sum_{|\alpha|=2}\frac{|D^{\alpha}u|}{|x|^2},\quad  \mbox{where}\quad C_3<\frac{3}{2}.
\end{align}
Borgne \cite{b2001} studied this problem for solutions to (\ref{his1}) additionally with the third order derivatives with the singular potential
$|x|^{-1+\varepsilon}$ $(\varepsilon>0)$.
 Colombini, Koch \cite{ck2010} studied the strong unique continuation  for solutions to the differential inequality (\ref{his1})  additionally with the third order derivatives with  the singular potential $\big|\ln|x|\big|^{-1-\varepsilon}|x|^{-1}$ for $\varepsilon>0$.
Recently, the present authors \cite{liu2022}  proved the  strong unique continuation to the following bi-Laplacian
\begin{equation*}
-\Delta^2 u+b_1(x)\cdot\nabla (\Delta u)+V_1(x)\Delta u+b_2(x)\cdot\nabla u+V_2(x)u=0,
\end{equation*}
with the strongly singular terms $b_1, b_2, V_1$ and $V_2$.
We refer to \cite{i2012,  linc2007, zhu2018} for the quantitative unique continuation properties of uniformly elliptic operators with higher order.

Before describing our main results, we will explain the definition of vanishing of infinite order in the subelliptic setting.
Usually, in subelliptic cases, a function $u\in L_{loc}^{2}$ is said to
vanish to infinite order at the origin if
\begin{align}\label{order}
\int_{B_r}u^2\psi dxdy=O(r^{k}), \quad \mbox{for every} \quad  k\in\mathbb{N},
\end{align}
as  $r\rightarrow 0$.
Moreover, we say that $u$ vanishes  at the origin more rapidly than any
power of $k$, namely that
\begin{align}\label{exp}
\int_{B_r}u^2\psi dxdy=O\left(\exp(-Br^{-\gamma})\right),
\end{align}
as $r\rightarrow0$ for some constants $B,\gamma>0$.

The main results of this paper are the following theorems.
\begin{theorem}\label{thm1}
Assume the potential $V$ satisfies (\ref{V}),  where $c_0$ is small in the sense of
\begin{align}\label{small-1}
\frac{4c_0}{(m-2)^2(Q-6)}+\frac{4}{(m-2)^4}<1,
\end{align}
or
\begin{align}\label{small-2}
\frac{4(c_0+1)}{(m-2)^2(Q-6)}<1,
\end{align}
where $Q=m+(\alpha+1)n$ denotes the homogenous dimension.
Let $u$ be  a weak solution of (\ref{L}).
Then there exist $C, A, \gamma>0$,  with $C$ and $\gamma$ only depending on $m$, $Q$, and $A$ depending on $u$, $m$, $Q$
such that for every ball $B_r$, for which $B_{2r}\subset\Omega$, it holds
\begin{align}\label{doubling}
\int_{B_{2r}}\left(u^2+(\Delta_{X}u)^2\right)\psi dxdy\leq C\exp\left(A r^{-\gamma}\right)\int_{B_{r}}\left(u^2+(\Delta_{X}u)^2\right)\psi dxdy.
\end{align}
\end{theorem}
Using the doubling estimate in Theorem \ref{thm1} above, we obtain the following strong unique continuation property for fourth order
Baouendi-Grushin type subelliptic operators with singular potential.
\begin{theorem}\label{thm0}
Let $\Omega$ be a  connected open subset of $\mathbb{R}^{m+n}$, with $m>4, Q>6$, containing the origin. Assume that $V$ satisfies the same assumptions of Theorem \ref{thm1}.
Let $u$ be a weak solution of (\ref{L}) in $\Omega$. If $u$ vanishes to infinite order at the origin in the sense of (\ref{exp}), then $u\equiv 0$ in $\Omega$.
\end{theorem}

\begin{remark}
With the help of the refined Hardy-Rellich inequalities (Lemma \ref{lem3}), similar to the proof of Theorem \ref{thm1}, for equation (\ref{second-order}), we can prove that \\
(i)\ if  $|V|\leq\frac{f(\rho)}{\rho^2}$ ($f$ is as in (\ref{f})) and $u$ vanishes to infinite order in the sense of (\ref{order}), then $u\equiv 0$;\\
(ii)\  if $V=V^{+}-V^{-}$, and $0<V^{+}\leq \frac{C}{\rho^2}$,  $0\leq V^{-}\leq\frac{\delta}{\rho^2}$ for some $\delta$ small,  assume that $u$ vanishes to infinite order in the sense of (\ref{exp}), then  $u\equiv 0$.\\
That is,   we can prove the strong unique continuation property  of (\ref{second-order}) under the case of removing
the weight $\psi$ in the assumptions (\ref{V2}) and $(\ref{V3})$ in \cite{G1993}.

\end{remark}

The argument we will adopt in this paper is to introduce an Almgren's type frequency function for the solutions $u$, and show its monotonicity  to obtain a doubling estimate. In order to prove the monotonicity of the frequency function and deal with the strongly singular potential such as  $\int_{B_r}\frac{Vuw}{|\nabla\rho|}$ with $w=\Delta_{X}u$, we should set up some refined Hardy-Rellich type inequalities for $\int_{B_r}\frac{u^2\psi}{\rho^6}$ and  $\int_{B_r}\frac{w^2}{\rho^2\psi}$ on the gauge balls.

On the other hand, the boundary term $\int_{\partial B_r}\frac{Vuw}{|\nabla\rho|}$ can not be directly controlled by $\int_{\partial B_r}\frac{u^2\psi}{|\nabla\rho|}$ and $\int_{\partial B_r}\frac{w^2\psi}{|\nabla\rho|}$ under  the assumption  (\ref{V}),  we  solve it by combining the divergence theorem and the higher order Hardy-Rellich inequalities which we established in  Section 3 (see (\ref{R6})).

The rest of the paper is organized as follows. In Section 2 we recall some notations and direct results  about  Baouendi-Grushin vector fields.
In order to control  the  strongly singular potentials, we prove some versions of Hardy-Rellich type inequalities on the gauge balls in Section 3.
In Section 4, we firstly introduce an Almgren's type frequency function  related to solutions of $Lu=0$, and prove the monotonicity property of it. Using the monotonicity property, we obtain a doubling estimate including the solutions $u$ as well as $\Delta_{X}u$. In Section 5,  we estimate a Caccioppoli type estimate to show that $\Delta_{X}u$  vanishes to infinite order provided $u$ does, and finally prove the strong unique continuous property.

\section{ Notations and  preliminary results}
We begin this section by giving some notations and some basic results about the Baouendi-Grushin type vector fields.

Let $(x,y)\in \mathbb{R}^{m}\times\mathbb{R}^{n}\equiv \mathbb{R}^{N}$.  Consider
the vector fields
\begin{equation}\label{vector}
X_i=\partial_ {x_i},
\hspace{1mm}i=1,\cdots,m,\hspace{2mm}
X_{m+j}=|x|^{\alpha} \partial_{y_j},\hspace{1mm}j=1,\cdots,n.
\end{equation}
We denote the Baouendi-Grushin gradient
and the corresponding divergence operator  as
\begin{align*}
&\nabla_{X}u=(X_1u,\cdots,X_{m+n}u), \quad \mbox{for a function}\  u, \\
&\mbox{div}_{X}F=\sum_{i=1}^{m+n}X_i F_i,\quad \mbox{for  a vector field}\ F=(F_1,\cdots, F_{m+n}),
\end{align*}
and the Baouendi-Grushin type Laplacian as
\begin{align}\label{Grushin}
\Delta_{X}u=\mbox{div}_{X}(\nabla_{X}u)=\Delta_{x}u+|x|^{2\alpha}\Delta_{y}u.
\end{align}
We note that the vector fields (\ref{vector}) are homogeneous of degree 1 with respect to the  anisotropic dilations
\begin{align}\label{dilation}
\delta_{\lambda}(x,y)=(\lambda x,\lambda^{\alpha+1} y),\quad  \lambda>0,
\end{align}
which leads to a homogeneous dimension  $Q=m+(\alpha+1)n$.  The infinitesimal generator of
the family of dilations (\ref{dilation}) is given by the vector field
\begin{equation}\label{Z0}
Z=\sum^{m}_{i=1}x_i\partial_{x_i}+(\alpha+1)\sum^{n}_{j=1}y_j\partial_{y_j}.
\end{equation}
Associated to the  vector
fields (\ref{vector}), for $(x,y)\in \mathbb{R}^{N}$,
there is a gauge norm given by
\begin{align}\label{gauge}
\rho(x,y)=\left(|x|^{2(\alpha+1)}+(\alpha+1)^2|y|^2\right)^{\frac{1}{2(\alpha+1)}}.
\end{align}
A gauge-ball and its gauge-sphere with respect to  $\rho$ centered at the origin with radius $r$ are
\begin{align*}
B_r=\{(x,y)\mid \rho(x,y)< r\}, \quad  \partial {B_r}=\{(x,y)\mid  \rho(x,y)=r\},
\end{align*}
respectively. Since $\rho\in C^{\infty}(\mathbb{R}^{N}\backslash\{(0,0)\})$, the outer unit normal on $\partial {B_r}$ is given by
$\vec{n}=\frac{\nabla\rho}{|\nabla\rho|}$,  where $\nabla\rho$ means the ordinary Euclidean gradient of $\rho$.
The gauge-ball  with respect to  $\rho$ centered at $(0,y_0)$ with radius $r$ is
\begin{align*}
B_{r}(0,y_0)=\left\{(x,y)\Large\mid  \left(|x|^{2(\alpha+1)}+(\alpha+1)^2|y-y_0|^2\right)^{\frac{1}{2(\alpha+1)}}<r\right\}.
\end{align*}
Introducing the angle function
\begin{align}\label{psi}
\psi\equiv |\nabla_{X}\rho|^2=\frac{|x|^{2\alpha}}{\rho^{2\alpha}}.
\end{align}
The function $\psi$ vanishes on the submanifold $\{0\}\times\mathbb{R}^{n}$ and clearly $0\leq\psi\leq1$.
In the following, we collect some important identities (see \cite{G1993}), which will  be used frequently in later sections.
\begin{align}\label{psi2}
\Delta_{X}\rho=\frac{Q-1}{\rho} \psi, \quad  \left|X_iX_j\rho\right|\leq \frac{C}{\rho}.
\end{align}
\begin{align}\label{Z}
Z\rho=\rho.
\end{align}
\begin{align}\label{Z2}
\nabla_{X}u\cdot \nabla_{X}\rho= \frac{Zu}{\rho}\psi.
\end{align}
\begin{align}\label{com}
[X_i,Z]=Z,\quad i=1,\cdots, m+n.
\end{align}

In order to study the weak solution to  equation (\ref{L}),
we introduce a function space associated with vector fields (\ref{vector1}), which is analogue to the corresponding Sobolev space \cite{Xu}. For
any integer $k\geq 1$, $p\geq 1$ and $\Omega\subset\mathbb{R^{N}}$, we define
\begin{align*}
M^{k,p}(\Omega)=\left\{ f\in L^{p}(\Omega) \mid X^{J}f\in L^{p}(\Omega), \forall J=(j_1,\cdots,j_s), |J|\leq k\right\}
\end{align*}
where $X^{J}f=X_{j_1}\cdots X_{j_s}f$ and define the norm in $M^{k,p}$ to be
\begin{align*}
\|f\|_{M^{k,p}(\Omega)}=\left(\sum_{|J|\leq k}\|X^{J}f\|^{p}_{L^{p}(\Omega)}\right)^{\frac{1}{p}}.
\end{align*}
We denote by $M_0^{k,p}(\Omega)$ the closure of $C_0^{\infty}(\Omega)$ in $M^{k,p}(\Omega)$.

A function $u\in M_{loc}^{2,2}(\Omega)$ is called a weak solution of (\ref{L}), it means
\begin{align*}
\int_{\Omega}\Delta_{X}u\Delta_{X}\phi=\int_{\Omega}Vu\phi,
\end{align*}
for any test function $\phi\in M_0^{2,2}(\Omega)$. In the above integral, and henceforth, we
omit indicating the Lebesgue measure  $dxdy$ in the relevant integrals.

In the rest of this section, we recall the  Sobolev embedding  inequality for the vector fields, and apply which to show  a equivalent  definition  of vanishing to infinite order in the sense of (\ref{exp}). The Sobolev embedding  inequality for the vector fields  (see \cite{M2006}) states
\begin{align}\label{sobolev}
\|u\|_{L^{2^{*}}(\Omega)}\leq C\|u\|_{M^{1,2}(\Omega)},
\end{align}
where $2^{*}=\frac{2Q}{Q-2}$ and constant $C=C(Q,\Omega)>0$.
\begin{proposition}\label{prop}
For any function $u\in M^{1,2}(B_r)$, then $u$ vanishes to infinite order  at the origin in
the sense of (\ref{exp}) if and only if it vanishes to infinite order in the following sense,
\begin{align}\label{exp2}
\int_{B_r}u^2 = O\left(\exp(-B_1r^{-\gamma})\right),
\end{align}
as $r\rightarrow0$ for  some constants  $B_1, \gamma>0$.
\end{proposition}
\begin{proof}
The proof is similar to  Lemma 2.4 in \cite {BM2020}. We give it here for the sake of
completeness.

Indeed, for some $q\in(2, 2^{*})$, H\"{o}lder inequality infers that
\begin{align}\label{q}
\int_{B_r}|u|\leq \left(\int_{B_r}u^q\psi\right)^{\frac{1}{q}} \left(\int_{B_r}\psi^{-\frac{1}{q-1}}\right)^{\frac{q-1}{q}}.
\end{align}
Noting that $\frac{1}{q-1}<1$, so $\frac{2\alpha}{q-1}<2$, which together with $m>2$ yields that for $r<1$,
\begin{align*}
\int_{B_r}\psi^{-\frac{1}{q-1}}dxdy\leq C(n)\int_{|x|\leq r}|x|^{-\frac{2\alpha}{q-1}}dx<\infty.
\end{align*}
By using an interpolation inequality and  (\ref{sobolev}), for $\frac{\theta}{2}+\frac{1-\theta}{2^{*}}=\frac{1}{q}$, it holds
\begin{align*}
\left(\int_{B_r}u^q\psi\right)^{\frac{1}{q}}&\leq
\left(\int_{B_r}u^2\psi\right)^{\frac{\theta}{2}}\left(\int_{B_r}u^{2^{*}}\psi\right)^{\frac{1-\theta}{2^{*}}}\\[2mm]
&\leq \left(\int_{B_r}u^2\psi\right)^{\frac{\theta}{2}} \|u\|^{1-\theta}_{M^{1,2}(B_r)}\\[2mm]
&\leq C e^{-\frac{\theta}{2}Br^{-\gamma}},
\end{align*}
by (\ref{exp}).
Therefore, (\ref{q}) implies
\begin{align*}
\int_{B_r}|u|\leq  C  e^{-\frac{\theta}{2}Br^{-\gamma}}.
\end{align*}
Using an interpolation inequality again,  for $\tau+\frac{1-\tau}{2^*}=\frac{1}{2}$, it holds
\begin{align*}
\|u\|_{L^2}\leq \|u\|^{\tau}_{L^1}\|u\|^{1-\tau}_{L^{2^{*}}}
\leq e^{-\tau\frac{\theta}{2}Br^{-\gamma}},
\end{align*}
that is (\ref{exp2}) with $B_1=\tau\theta B$.
\end{proof}

\section{Some Hardy-Rellich inequalities  in gauge balls}

In this section, we prove some local versions of Hardy-Rellich type inequalities related to vector fields (\ref{vector}) on  gauge balls. They will play an important role in showing doubling estimates  and are also independently interesting themselves.

There are many versions of Hardy type inequalities related the  vector fields (see e.g., \cite{D2004, hardy2004, R2017, 2019PAMS}). However, those inequalities are established for functions with compact support. What we need is a local version.
In \cite{GL} the authors obtained the following  inequality for the Heisenberg group, which can be proved for  Baouendi-Grushin  vector fields (\ref{vector}),
\begin{align*}
\int_{B_r}\frac{ u^2\psi}{\rho^2}\leq \left(\frac{2}{Q-2}\right)^2\int_{B_r}|\nabla_{X}u|^2+\frac{2}{Q-2}r^{-1}\int_{\partial B_r}\frac{u^2\psi}{|\nabla\rho|}.
\end{align*}

Our task is to work under the assumption
$|V|\leq \frac{C}{\rho^4}$ for $(\ref{L})$,  we need to estimate some terms
or quantities  such as  $\int_{B_r}\frac{ u^2}{\rho^2\psi}$ and $\int_{B_r}\frac{u^2\psi}{\rho^6}$.
Let us start from the following  refined  Hardy inequality for the vector fields  (\ref{vector}) in a gauge ball.
\begin{lemma} \label{lem3}
Assume $m>2$, for every $u\in M^{1,2}(B_r)$, it holds
\begin{align}\label{hhardy-1}
\int_{B_r}\frac{ u^2}{|x|^2}\leq \left(\frac{2}{m-2}\right)^2\int_{B_r}|\nabla_{X}u|^2+\frac{2}{m-2}r^{-1}\int_{\partial B_r}\frac{u^2\psi}{|\nabla\rho|}.
\end{align}
Moreover,  if $0<\alpha\leq1$, then the following inequality is valid,
\begin{align}\label{hhardy-2}
\int_{B_r}\frac{ u^2}{\rho^2\psi}\leq \left(\frac{2}{m-2}\right)^2\int_{B_r}|\nabla_{X}u|^2+\frac{2}{m-2}r^{-1}\int_{\partial B_r}\frac{u^2\psi}{|\nabla\rho|}.
\end{align}
\end{lemma}
\begin{proof}
Let
\begin{equation*}
h=\frac{1}{|x|^2}
\left(
\begin{array}{lll}
x \\
0
\end{array}\right)
\in \mathbb{R}^{m+n}
\end{equation*}
 A direct calculation shows
that
\begin{equation}\label{h-1}
\mbox{div}_{X}h=\frac{m-2}{|x|^2}, \hspace{3mm} \quad
h\cdot \nabla_{X}\rho=\rho^{-(2\alpha+1)}|x|^{2\alpha}=\rho^{-1}\psi.
\end{equation}
Hence, applying integration by parts, using (\ref{h-1}), and the Cauchy inequality, one has
\begin{align*}
\int_{B_r}\frac{ u^2}{|x|^2}&=\frac{1}{m-2}\int_{B_r}\mbox{div}_{X}h u^2 \\[2mm]
&=-\frac{1}{m-2}\int_{B_r}\nabla_{X}u^2\cdot h +\frac{1}{m-2}\int_{\partial B_r}u^2 h\cdot\frac{\nabla_{X}\rho}{|\nabla\rho|}\\[2mm]
&\leq\frac{2}{m-2}\int_{B_r}\frac{|u||\nabla_{X}u|}{|x|}+\frac{1}{m-2}r^{-1}\int_{\partial B_r}\frac{u^2\psi}{|\nabla\rho|} \\[2mm]
&\leq\frac{1}{2}\int_{B_r}\frac{ u^2}{|x|^2}+\frac{2}{(m-2)^2}\int_{B_r}|\nabla_Xu|^2+\frac{1}{m-2}r^{-1}\int_{\partial B_r}\frac{u^2\psi}{|\nabla\rho|},
\end{align*}
which yields the desired result (\ref{hhardy-1}).
Furthermore, if  $\alpha\leq 1$, then $\rho^2\psi\geq |x|^2$, so (\ref{hhardy-1}) yields (\ref{hhardy-2}) directly.
\end{proof}

\begin{lemma}\label{lem4}
Assume $Q>6$. Then, for any $u\in M^{3,2}(B_r)$, it holds
\begin{align}\label{ineq6-1}
\int_{B_r} \frac{u^2\psi}{\rho^6}\leq \frac{1}{(Q-6)^2}\int_{B_r}\frac{|\Delta_{X}u|^2}{\rho^2\psi}+\frac{2}{Q-6}r^{-5}\int_{\partial B_r}\frac{u^2\psi}{|\nabla\rho|}.
\end{align}
Moreover, if $m>2$ and $0<\alpha\leq1$, it holds
\begin{align}\label{ineq6-2}
\int_{B_r} \frac{u^2\psi}{\rho^6}&\leq \frac{4}{(m-2)^2(Q-6)^2}\int_{B_r}|\nabla_X(\Delta_{X}u)|^2\nonumber\\[2mm]
&+\frac{2}{(m-2)(Q-6)^2}r^{-1}\int_{\partial B_r} \frac{(\Delta _{X}u)^2\psi}{|\nabla\rho|}+\frac{2}{Q-6}r^{-5}\int_{\partial B_r}\frac{u^2\psi}{|\nabla\rho|}.
\end{align}
\end{lemma}
\begin{proof}
Noting
\begin{equation*}
\Delta_{X}\rho^{-4}=-4(Q-6)\rho^{-6}\psi.
\end{equation*}
Integrating by parts, one obtain
\begin{align} \label{6-1}
\int_{B_r}\frac{u^2\psi}{\rho^6}&=-\frac{1}{4(Q-6)}\int_{B_r}\Delta_{X}\left(\frac{1}{\rho^4}\right)u^2\nonumber\\[2mm]
&=-\frac{1}{4(Q-6)}\int_{B_r}\frac{1}{\rho^4} \Delta_{X}u^2+\frac{1}{4(Q-6)}\int_{\partial B_r}\frac{1}{\rho^4}\nabla_{X}u^2\cdot\frac{\nabla_{X}\rho}{|\nabla\rho|}\nonumber\\[2mm]
&-\frac{1}{4(Q-6)}\int_{\partial B_r}u^2\nabla_{X}\left(\frac{1}{\rho^4}\right)\cdot\frac{\nabla_{X}\rho}{|\nabla\rho|}\nonumber\\[2mm]
&\equiv J_1+J_2+J_3.
\end{align}
Since $Q>6$, the term $J_1$ can be estimated as
\begin{align}\label{J1}
J_1 &=-\frac{1}{2(Q-6)}\int_{B_r}\frac{u\Delta_{X}u}{\rho^4}-\frac{1}{2(Q-6)}\int_{B_r}\frac{|\nabla_{X}u|^2}{\rho^4}\nonumber\\[2mm]
&\leq \frac{1}{2(Q-6)}\int_{B_r}\frac{|u||\Delta_{X}u|}{\rho^4}-\frac{1}{2(Q-6)}\int_{B_r}\frac{|\nabla_{X}u|^2}{\rho^4},
\end{align}
By using the divergence theorem,
\begin{align}\label{J2}
J_2 &=\frac{1}{2(Q-6)}r^{-4}\int_{\partial B_r} u \frac{\nabla_Xu\cdot\nabla_{X}\rho}{|\nabla\rho|}\nonumber\\[2mm]
&=\frac{1}{2(Q-6)}r^{-4}\int_{B_r}\mbox{div}_{X}\left(u\nabla_Xu\right)\nonumber\\[2mm]
&=\frac{1}{2(Q-6)}r^{-4}\int_{B_r}u\Delta_{X}u + \frac{1}{2(Q-6)}r^{-4}\int_{B_r}|\nabla_{X}u|^2\nonumber\\[2mm]
&\leq \frac{1}{2(Q-6)}\int_{B_r}\frac{|u||\Delta_{X}u|}{\rho^4}+\frac{1}{2(Q-6)}\int_{B_r}\frac{|\nabla_{X}u|^2}{\rho^4},
\end{align}
where we have used  $Q>6$ again in the last inequality, then the last term on the right-hand side of (\ref{J2}) can cancel  out that of (\ref{J1}).
The term $J_3$ becomes
\begin{align}\label{J3}
J_3 &=\frac{1}{(Q-6)}\int_{\partial B_r} u^2 \rho^{-5} \frac{|\nabla_X\rho|^2}{|\nabla\rho|}\nonumber\\[2mm]
&=\frac{1}{(Q-6)}r^{-5}\int_{\partial B_r} u^2 \frac{\psi}{|\nabla\rho|}.
\end{align}
Substituting  (\ref{J1})-(\ref{J3}) into (\ref{6-1}),
 \begin{align*}
\int_{B_r}\frac{u^2\psi}{\rho^6} &\leq \frac{1}{Q-6}\int_{B_r}\frac{|u||\Delta_{X}u|}{\rho^4}+ \frac{1}{Q-6}r^{-5}\int_{\partial B_r} u^2\frac{\psi}{|\nabla\rho|}\nonumber\\[2mm]
&\leq \frac{1}{2} \int_{B_r}\frac{u^2\psi}{\rho^6}+ \frac{1}{2(Q-6)^2}\int_{B_r}\frac{|\Delta_{X}u|^2}{\rho^2\psi}+\frac{1}{Q-6}r^{-5}\int_{\partial B_r} u^2\frac{\psi}{|\nabla\rho|},
\end{align*}
which deduces (\ref{ineq6-1}).

Furthermore, noting that $\rho^2\psi\geq |x|^2$ thanks to $\alpha\leq 1$, then we can apply Lemma \ref{lem3} for the first term on the right-hand side of (\ref{ineq6-1}) to deduce (\ref{ineq6-2}).
\end{proof}

\begin{lemma} \label{lem5}
Assume $m>2$, $Q>6$  and $\alpha\leq1$,  and let $u\in M^{3,2}(\Omega)$. Then
\begin{align}\label{hardy-4}
\int_{B_r} \frac{|\nabla_{X}u|^2}{\rho^4}\leq C\left(r^{-4}\int_{B_r}|\nabla_{X} u|^2+\int_{B_r}|\nabla_{X}(\Delta_{X}u)|^2+r^{-5}\int_{\partial B_r}\frac{u^2\psi}{|\nabla\rho|}+r^{-1}\int_{\partial B_r}\frac{(\Delta _{X}u)^2\psi}{|\nabla\rho|}\right),
\end{align}
and
\begin{align} \label{hhardy2}
&\int_{B_r}\frac{ u^2}{\rho^6}\leq\int_{B_r}\frac{ u^2}{\rho^4|x|^2}\nonumber\\[2mm]
&\leq C\left( \int_{B_r}|\nabla_X(\Delta_{X}u)|^2+r^{-4}\int_{B_r}|\nabla_{X} u|^2+r^{-1}\int_{\partial B_r} \frac{(\Delta _{X}u)^2\psi}{|\nabla\rho|}+r^{-5}\int_{\partial B_r}\frac{u^2\psi}{|\nabla\rho|}\right).
 \end{align}
\end{lemma}
\begin{proof}
Using integration by parts and the Cauchy inequality, we have
\begin{align*}
\int_{B_r}\frac{|\nabla_{X} u|^2}{\rho^4}&=-\int_{B_r}\mbox{div}_{X}\left(\frac{\nabla_{X} u}{\rho^4}\right)u+r^{-4}\int_{\partial B_r} u \frac{\nabla_{X}u\cdot\nabla_{X}\rho}{|\nabla\rho|}\nonumber\\[2mm]
&=-\int_{B_r}\frac{u \Delta_{X} u}{\rho^4}+5\int_{B_r}\rho^{-5}u\nabla_{X} u\cdot\nabla_{X}\rho+r^{-4}\int_{B_r}\mbox{div}_{X}\left(u\nabla_{X} u\right)\nonumber\\[2mm]
&=-\int_{B_r}\frac{u \Delta_{X} u}{\rho^4}+5\int_{B_r}\rho^{-5}u\nabla_{X} u\cdot\nabla_{X}\rho+r^{-4}\int_{B_r}|\nabla_{X} u|^2
+r^{-4}\int_{B_r} u\Delta_{X} u \nonumber\\[2mm]
&\leq 2 \int_{B_r}\frac{|u| |\Delta_{X} u|}{\rho^4} +\varepsilon \int_{B_r}\frac{|\nabla_{X} u|^2}{\rho^4}+C(\varepsilon)\int_{B_r}\frac{u^2|\nabla_{X}\rho|^2}{\rho^6}+ r^{-4}\int_{B_r}|\nabla_{X} u|^2\\[2mm]
&\leq \int_{B_r}\frac{u^2\psi}{\rho^6} +\int_{B_r}\frac{|\Delta_{X}u|^2}{\rho^2\psi}+\varepsilon \int_{B_r}\frac{|\nabla_{X} u|^2}{\rho^4}+C(\varepsilon)\int_{B_r}\frac{u^2\psi}{\rho^6}+ r^{-4}\int_{B_r}|\nabla_{X} u|^2,
\end{align*}
which yields
\begin{align*}
\int_{B_r}\frac{|\nabla_{X} u|^2}{\rho^4}
\leq C\left(r^{-4}\int_{B_r}|\nabla_{X} u|^2 +\int_{B_r}\frac{|\Delta_{X}u|^2}{\rho^2\psi}+\int_{B_r}\frac{u^2\psi}{\rho^6}\right).
\end{align*}
Applying Lemma \ref{lem3} and Lemma \ref{lem4} implies the desired result (\ref{hardy-4}).

Using the same function $h$ in Lemma \ref{lem3} and using (\ref{h-1}) again, integrating by parts,  one has
\begin{align*}
\int_{B_r}\frac{ u^2}{\rho^4|x|^2}&=\frac{1}{m-2}\int_{B_r}\mbox{div}_{X}h \frac{u^2}{\rho^4} \\[2mm]
&=-\frac{1}{m-2}\int_{B_r}\nabla_{X}(\rho^{-4}u^2)\cdot h +\frac{1}{m-2}\int_{\partial B_r}\frac{u^2}{\rho^4} h\cdot\frac{\nabla_{X}\rho}{|\nabla\rho|}\\[2mm]
&=-\frac{2}{m-2}\int_{B_r}\frac{u\nabla_{X}u\cdot h}{\rho^4}+ \frac{4}{m-2}\int_{B_r}\frac{u^2\nabla_{X}\rho\cdot h}{\rho^5}+\frac{1}{m-2}r^{-5}\int_{\partial B_r}\frac{u^2\psi}{|\nabla\rho|} \\[2mm]
&\leq\frac{1}{2} \int_{B_r}\frac{ u^2}{\rho^4|x|^2}+\frac{2}{(m-2)^2}\int_{B_r}\frac{|\nabla_{X}u|^2}{\rho^4}+ \frac{4}{m-2}\int_{B_r}\frac{u^2\psi}{\rho^6}+\frac{1}{m-2}r^{-5}\int_{\partial B_r}\frac{u^2\psi}{|\nabla\rho|}.
\end{align*}
Therefore,
\begin{align*}
\int_{B_r}\frac{ u^2}{\rho^4|x|^2}
\leq \frac{4}{(m-2)^2}\int_{B_r}\frac{|\nabla_{X}u|^2}{\rho^4}+ \frac{8}{m-2}\int_{B_r}\frac{u^2\psi}{\rho^6}+\frac{2}{m-2}r^{-5}\int_{\partial B_r}\frac{u^2\psi}{|\nabla\rho|},
\end{align*}
which together with Lemma \ref{lem4} and Lemma \ref{lem5} yields the estimate (\ref{hhardy2}).
\end{proof}

\begin{remark}
(i)\ The integrands on the the boundary integrals of these inequalities have the weight function $\psi$.
This is an important point and will be shown in the next section when we prove the monotonicity of the frequency function.
 \\
 (ii)\  All the local Hardy-Rellich type inequalities in this Section also hold on the  Heisenberg-type groups.
\end{remark}

\section{The frequency function and doubling estimate }
In this section, we first introduce an Almgren's type frequency function and study the behavior of it.  After obtaining  a monotonicity of the frequency function,  we establish a doubling estimate  for a combination of  $u$  and  $\Delta_{X}u$.

First,
 we decompose the fourth order equation (\ref{L}) into a system of two second order equations, that is,
\begin{equation}\label{w}
\left\{\begin{array}{lll}
\Delta_{X} u=w,\\[2mm]
\Delta_{X} w=V u.
\end{array}\right.
\end{equation}
Let $u\in M_{loc}^{2,2}(\Omega)$ be a weak solution of (\ref{L}).
Define
\begin{equation*}
H_1(r)=\int_{\partial B_r}u^2\frac{\psi}{|\nabla\rho|},\quad \quad H_2(r)=\int_{\partial B_r}w^2\frac{\psi}{|\nabla\rho|},
\end{equation*}
and
\begin{equation}\label{h}
H(r)=H_1(r)+r^{4}H_2(r).
\end{equation}
\begin{lemma}\label{lemma-H}
Let $u$ be a nonzero solution of (\ref{L}), where  $V$ satisfies the assumption  (\ref{V}), with $c_0$ is small in the sense of
\begin{align}\label{small-0}
\frac{4c_0}{(m-2)^2(Q-6)}<1.
\end{align}
Then there exists $r_0$, such that
\begin{equation*}
H(r)\neq 0\quad \mbox{for every}\quad r\in(0,r_0).
\end{equation*}
\end{lemma}
\begin{remark}
We observe that, if $u\in M_{loc}^{2,2}(\Omega)$ is a weak solution of (\ref{L}), then
we can check that $u\in M_{loc}^{3,2}(\Omega)$ provided $c_0$ is small in the sense (\ref{small-0}) by using the difference quotient method  and the Hardy-Rellich inequality.
\end{remark}
\begin{proof}
We prove it by contradiction. Suppose that $H(r_{*})=0$ for some $r_{*}\in(0,r_0)$. Then the definition of $H(r)$ implies that $u|_{\partial B_{r_{*}}}=0$ and $w|_{\partial B_{r_*}}=0$.

Multiplying the second equation of (\ref{w}) by $w$ and then integrating over $B_{r_*}$, integrating by parts
with $w|_{\partial B_{r_*}}=0$, we get
\begin{align*}
\int_{B_{r_*}}Vuw=\int_{B_{r_*}}\Delta_X w w=-\int_{B_{r_*}}|\nabla_X w|^2.
\end{align*}
Then, by using the  Hardy-Rellich inequalities (\ref{hhardy-2}) and (\ref{ineq6-2}) with $u|_{\partial {B_{r_*}}}=w|_{\partial {B_{r_*}}}=0$, we have
\begin{align*}
\int_{B_{r_*}}|\nabla_X w|^2&\leq \int_{B_{r_*}}|V||u||w|\leq c_0\int_{B_{r_*}}\frac{|u||w|}{\rho^4}\\[2mm]
&\leq\varepsilon \int_{B_{r_*}}\frac{u^2\psi}{\rho^6}+ \frac{1}{4}c_0^2\varepsilon^{-1}\int_{B_{r_*}}\frac{w^2}{\rho^2\psi}\\[2mm]
&\leq  \frac{4}{(m-2)^2(Q-6)^2}\varepsilon r_*^{4}\int_{B_{r_*}}|\nabla_{X}w|^2
+\frac{4}{(m-2)^2}\frac{1}{4}c_0^2\varepsilon ^{-1}r_*^{4}\int_{B_{r_*}}|\nabla_{X}w|^2\\[2mm]
&=\left(\frac{4}{(m-2)^2(Q-6)^2}\varepsilon+\frac{4}{(m-2)^2}\frac{1}{4}c_0^2\varepsilon ^{-1}\right)r_*^{4}\int_{B_{r_*}}|\nabla_{X}w|^2.
\end{align*}
Let
\begin{align*}
g(\varepsilon)\equiv\frac{4}{(m-2)^2(Q-6)^2}\varepsilon+ \frac{4}{(m-2)^2}\frac{1}{4}c_0^2\varepsilon ^{-1},
\end{align*}
it is easy to see that
\begin{align*}
g_{min}=\frac{4c_0}{(m-2)^2(Q-6)},
\end{align*}
which deduces
\begin{align*}
\int_{B_{r_*}}|\nabla_X w|^2\leq \frac{4c_0}{(m-2)^2(Q-6)} \int_{B_{r_*}}|\nabla_X w|^2.
\end{align*}
This is a contradiction thanks  to the assumption (\ref{small-0}) under the case $\int_{B_{r_*}}|\nabla_X w|^2\neq0$.
If $\int_{B_{r_*}}|\nabla_X w|^2=0$, then applying  Sobolev-Poincar\'{e} inequalities related to the vector fields (\ref{vector1}) (see, e.g.,\cite{FG1994}), one has $w=0$ in $B_{r_*}$. That is
\begin{align*}
\left\{
\begin{array}{lll}
\Delta_{X}u=0,  \quad \mbox{in}\ B_{r_*},\\[2mm]
u\mid_{\partial{B_{r_*}}}=0.
\end{array}
\right.
\end{align*}
Applying the maximum principle for Baouendi-Grushin operator \cite{max}, there holds $u=0$ in $B_{r_*}$. This  contradicts  the hypothesis of a non-zero solution.
The proof of Lemma \ref{lemma-H} is concluded.
\end{proof}

Define
\begin{equation}\label{I3}
I_1(r)=\int_{ B_r} |\nabla_X u|^2+\int_{B_r}uw,
\quad \quad   I_2(r)=\int_{ B_r} |\nabla_X w|^2+\int_{B_r} V wu,
\end{equation}
and
\begin{equation}\label{I}
I(r)=I_1(r)+ r^4 I_2(r).
\end{equation}
Moreover, by using the divergence theorem, equations (\ref{w}) and the fact (\ref{Z2}),  $I_1(r)$ and $I_2(r)$ can be rewritten as
\begin{align}\label{I1}
&I_1(r)=\int_{\partial B_r}u \frac{\nabla_{X}u\cdot\nabla_{X}\rho}{|\nabla\rho|}=\frac{1}{r}\int_{\partial B_r}\frac{u Zu \psi}{|\nabla\rho|},\nonumber\\[2mm]
&I_2(r)=\int_{\partial B_r}w \frac{\nabla_{X}w\cdot\nabla_{X}\rho}{|\nabla\rho|}=\frac{1}{r}\int_{\partial B_r}\frac{w Zw \psi}{|\nabla\rho|}.
\end{align}

\begin{lemma}\label{lemma1}
Let $u$ be a solution of (\ref{L}). Then
\begin{align}\label{H'(r)}
H'(r)=\frac{Q-1}{r}H(r)+2I(r)+4r^3H_{2}(r).
\end{align}
\end{lemma}
\begin{proof}
By using the divergence theorem and fact (\ref{psi}), one has
\begin{align*}
H_1(r)&=\int_{\partial B_r}u^2 \frac{\nabla_{X}\rho\cdot\nabla_{X}\rho}{|\nabla\rho|}
=\int_{B_r}\mbox{div}_{X}\left(u^2\nabla_{X}\rho\right)\\[2mm]
&=\int_{B_r} u^2\Delta_{X}\rho+2\int_{B_r} u\nabla_{X}u\cdot\nabla_{X}\rho.
\end{align*}
Using the co-area formula, we get
\begin{equation}\label{h}
H_1(r)=\int_{0}^{r}ds\int_{\partial
B_{s}}u^2\frac{\Delta_{X}\rho}{\left|\nabla\rho\right|}+2\int_{0}^{r}ds\int_{\partial
B_s} u \frac{ \nabla_{X}u\cdot\nabla_{X}\rho}{\left|\nabla\rho\right|}
\end{equation}
Differentiating (\ref{h}) with respect to $r$ and according to
(\ref{psi2}), (\ref{I1}), we get
\begin{align}\label{H1}
H_1'(r)&=\int_{\partial
B_r}u^2\frac{\Delta_{X}\rho}{|\nabla\rho|}+2\int_{\partial
B_r}u\frac{\nabla_{X}u\cdot\nabla_{X}\rho}{|\nabla
\rho|}\nonumber\\[2mm]
&=\frac{Q-1}{r}H_1(r)+2I_1(r).
\end{align}
In a similar way, one has
\begin{align}\label{H2}
H_2'(r)=\frac{Q-1}{r}H_2(r)+2 I_2(r).
\end{align}
Combining (\ref{H1}) and (\ref{H2}) implies the Lemma.
\end{proof}

Due to Lemma \ref{lemma-H}, we can define the frequency function as
\begin{equation}\label{n}
N(r)=\frac{rI(r)}{H(r)}, \quad \mbox{for every}\quad r\in(0,r_0).
\end{equation}
Lemma \ref{lemma-H} also implies that the function $r\rightarrow N(r)$ is absolutely continuous on $(0,r_0)$. Let
\begin{align}\label{ome}
\Omega_{r_0}=\left\{r\in(0,r_0):N(r)>\max(1, N(r_0))\right\}.
\end{align}
Then $\Omega_{r_0}$ is an open subset of $\mathbb{R}$. Therefore,
\begin{align}\label{omega}
\Omega_{r_0}=\bigcup_{j=1}^{\infty}(a_j,b_j),\quad a_j,b_j\notin \Omega_{r_0},
\end{align}
and
\begin{align}\label{h<i}
\frac{H(r)}{r}<I(r),\quad \mbox{for all} \quad r\in \Omega_{r_0}.
\end{align}

The main purpose of this section is to prove the doubling estimate of solutions to (\ref{L}) (Theorem \ref{thm1}). To achieve this goal, we should build the monotonicity of the frequency function. Precisely,
\begin{theorem}\label{thm4}
Let $u$ be a solution of (\ref{L}), where  $V$ satisfies the assumption  (\ref{V}), with $c_0$ is small in the sense of (\ref{small-1}) or (\ref{small-2}).
There exists a constant $\beta>0$, depending only on $m, Q$ and $c_0$,  such that
for every $r\in(0,r_0)$ for which (\ref{h<i}) holds,  we have
\begin{align}\label{key6}
\frac{N'(r)}{N(r)}\geq -\beta\frac{1}{r}.
\end{align}
\end{theorem}

We postpone  the proof of Theorem \ref{thm4} for the moment and are first devoted to set up  a doubling estimate for solutions $u$  to (\ref{L}). This proof is standard (see, e.g., \cite{GL}),  we include it for the sake of completeness.

{\bf Proof of Theorem \ref{thm1}}.\
Integrating  (\ref{key6}) on $(r,b_j)$ where $r\in(a_j,b_j)$, we get
\begin{equation*}
\ln N(b_j)-\ln N(r)\geq -\beta(\ln b_j-\ln r),
\end{equation*}
which yields
\begin{equation*}
N(r)\leq N(b_j)b_j^{\beta}r^{-\beta}.
\end{equation*}
Recalling that $N(b_{j})\leq \max(1,N(r_0))$, we infer that
\begin{equation*}\label{9}
N(r)\leq C\max(1,N(r_0))r^{-\beta}, \quad \mbox{for}\quad r\in\Omega_{r_0}.
\end{equation*}
By the  definitions of $H(r)$ and $I(r)$, one deduces from  (\ref{H'(r)}) that
\begin{equation}\label{ineq}
\left(\ln\frac{H(r)}{r^{Q+3}}\right)'\leq \frac{2N(r)}{r}.
\end{equation}
Integrating (\ref{ineq}) from $r$ to $2r$, with $r\leq \frac{r_0}{2}$, we deduce
\begin{align}\label{10}
\ln\frac{H(2r)}{(2r)^{Q+3}}-\ln\frac{H(r)}{r^{Q+3}}\leq 2\int_{r}^{2r}\frac{N(t)}{t}dt.
\end{align}
Let $J_{r}=\{t\in(r,2r): t\notin\Omega_{r_0}, N(t)\geq0\}$, where $\Omega_{r_0}$ is as in (\ref{ome}). Then
\begin{align}\label{est}
\int_{r}^{2r}\frac{N(t)}{t}dt&\leq \int_{(r,2r)\cap\Omega_{r_0}}\frac{N(t)}{t}dt+\int_{J_r}\frac{N(t)}{t}dt\nonumber\\[2mm]
&\leq C\max(1,N(r_0))\int_{r}^{2r}t^{-\beta+1}dt+\max(1,N(r_0))\int_{r}^{2r}\frac{1}{t}dt\nonumber\\[2mm]
&\leq C\max(1,N(r_0))r^{-\beta+2}.
\end{align}
Plugging (\ref{est}) into (\ref{10}), it follows that
\begin{align}\label{dou2}
H(2r)\leq 2^{Q+3}\exp\left(C\max(1,N(r_0))r^{-\beta+2}\right)H(r).
\end{align}
Integrating (\ref{dou2}) from $\frac{r}{2}$ to $r$, we derive the desired estimate (\ref{doubling}) with $\gamma=\beta-2$ and $A=C\max(1,N(r_0))$.\qed

Now, we turn to prove the monotonicity of $N(r)$. We first estimate $I'(r)$ in  the following lemma, which is the key step to show $N(r)$'s monotonicity. The refined  Hardy-Rellich type inequalities  plays a very important role in this calculation.

\begin{lemma}\label{lemma2}
Let $u$ be a solution of (\ref{L}), where  $V$ satisfies the assumption  (\ref{V}), with $c_0$ is
small in the sense of (\ref{small-1}) or (\ref{small-2}).
Then for every $r\in(0,r_0)$ for which (\ref{h<i}) holds,  the following estimate is valid,
\begin{align}\label{I'(r)}
I'(r)=\frac{Q-2}{r}I(r)+\frac{2}{r^2}\int_{\partial{B_r}}\frac{(Zu)^2\psi}{|\nabla\rho|}
+2r^2\int_{\partial{B_r}}\frac{(Zw)^2\psi}{|\nabla\rho|}+O(r^{-1})I(r).
\end{align}
\end{lemma}
\begin{proof}
At first, we  prove the following assertion:
if $c_0$ is small  in the sense of (\ref{small-1}) or (\ref{small-2}), then it holds
\begin{align}\label{claim}
\int_{B_r}|\nabla_{X}u|^2+r^{4}\int_{B_r}|\nabla_{X}w|^2\leq CI(r).
\end{align}
Indeed, by the definition of $I(r)$, one has
\begin{align}\label{claim-0}
\int_{B_r}|\nabla_{X}u|^2+r^{4}\int_{B_r}|\nabla_{X}w|^2=I(r)-\int_{B_r}uw-r^4\int_{B_r}Vuw.
\end{align}
Using the Cauchy's inequality with $\varepsilon$ as in Lemma \ref{lemma-H}, and applying Lemmas \ref{lem3}, \ref{lem4},  we obtain
\begin{align}\label{claim-1}
r^4\int_{B_r}Vuw&\leq c_0 r^4\int_{B_r}\frac{|u||w|}{\rho^4}
\leq \varepsilon  r^4\int_{B_r}\frac{u^2\psi}{\rho^6}+ \frac{1}{4}c_0^2\varepsilon ^{-1}r^4\int_{B_r}\frac{w^2}{\rho^2\psi}\nonumber\\[2mm]
&\leq \frac{4}{(m-2)^2(Q-6)^2}\varepsilon r^{4}\int_{B_r}|\nabla_{X}w|^2
+\frac{4}{(m-2)^2}\frac{1}{4}c_0^2\varepsilon ^{-1}r^{4}\int_{B_r}|\nabla_{X}w|^2\nonumber\\[2mm]
&+Cr^{-1}\int_{\partial B_r}\frac{u^2\psi}{|\nabla\rho|}+Cr^{3}\int_{\partial B_r}\frac{w^2\psi}{|\nabla\rho|}\nonumber\\[2mm]
&\leq \frac{4c_0}{(m-2)^2(Q-6)}r^4\int_{B_r}|\nabla_{X}w|^2+Cr^{-1}\left(\int_{\partial B_r}\frac{u^2\psi}{|\nabla\rho|}+r^{4}\int_{\partial B_r}\frac{w^2\psi}{|\nabla\rho|}\right).
\end{align}
For the term $\int_{B_r}uw$, we can  estimate as  above with $c_0=1$, that is,
\begin{align}\label{claim-2}
\int_{B_r}uw&\leq r^4\int_{B_r}\frac{|u||w|}{\rho^4}\nonumber\\[2mm]
&\leq \frac{4}{(m-2)^2(Q-6)}r^4\int_{B_r}|\nabla_{X}w|^2+Cr^{-1}\left(\int_{\partial B_r}\frac{u^2\psi}{|\nabla\rho|}+r^{4}\int_{\partial B_r}\frac{w^2\psi}{|\nabla\rho|}\right).
\end{align}
Putting (\ref{claim-1}), (\ref{claim-2}) into (\ref{claim-0}),  and using (\ref{h<i}) and  the assumption (\ref{small-2}), it holds
\begin{align*}
&\int_{B_r}|\nabla_{X} u|^2+r^{4}\int_{B_r}|\nabla_{X} w|^2\\[2mm]
&\leq  \frac{4c_0+4}{(m-2)^2(Q-6)}r^4\int_{B_r}|\nabla_{X}w|^2+Cr^{-1}\left(\int_{\partial B_r}\frac{u^2\psi}{|\nabla\rho|}+r^{4}\int_{\partial B_r}\frac{w^2\psi}{|\nabla\rho|}\right)\\[2mm]
&<r^4\int_{B_r}|\nabla_{X}w|^2+CI(r),
\end{align*}
so, the claim (\ref{claim}) is proved  provided  (\ref{small-2}).

On the other hand, applying Lemma \ref{lem3} to $u$ and $w$, the term $\int_{B_r}uw$ can also be controlled as
\begin{align*}
\int_{B_r}uw&\leq r^2\int_{B_r}\frac{|u||w|}{\rho^2}\leq \frac{\epsilon}{2}\int_{B_r}\frac{u^2}{\rho^2}+\frac{1}{2\epsilon}r^4\int_{B_r}\frac{w^2}{\rho^2}\nonumber\\[2mm]
&\leq \frac{2\epsilon}{(m-2)^2}\int_{B_r}|\nabla_{X}u|^2 +\frac{2}{(m-2)^2\epsilon}r^4\int_{B_r}|\nabla_{X}w|^2+Cr^{-1}\left(\int_{\partial B_r}\frac{u^2\psi}{|\nabla\rho|}+r^4\int_{\partial B_r}\frac{w^2\psi}{|\nabla\rho|}\right).
\end{align*}
Choosing $\epsilon$ such that $\frac{2\epsilon}{(m-2)^2}<1$, i.e. $\epsilon=\frac{(m-2)^2}{2+\delta}$  for some $\delta>0$ to be  determined later. Then the above inequality becomes
\begin{align}\label{claim-3}
\int_{B_r}uw
\leq \frac{2}{2+\delta}\int_{B_r}|\nabla_{X}u|^2 +\frac{4+2\delta}{(m-2)^4}r^4\int_{B_r}|\nabla_{X}w|^2+Cr^{-1}H(r).
\end{align}
Substituting (\ref{claim-1}) and (\ref{claim-3}) into (\ref{claim-0}), and using (\ref{h<i}),
\begin{align}\label{claim-4}
&\int_{B_r}|\nabla_{X} u|^2+r^{4}\int_{B_r}|\nabla_{X} w|^2\nonumber\\[2mm]
&\leq \frac{2}{2+\delta}\int_{B_r}|\nabla_{X}u|^2 +\left(\frac{4+2\delta}{(m-2)^4}+\frac{4c_0}{(m-2)^2(Q-6)}\right)r^4\int_{B_r}|\nabla_{X}w|^2+ CI(r).
\end{align}
Since (\ref{small-1}),  let $\delta=\frac{1}{4}(m-2)^4\left(1-\frac{4c_0}{(m-2)^2(Q-6)}-\frac{4}{(m-2)^4}\right)>0$, then
\begin{align*}
\frac{4+2\delta}{(m-2)^4}+\frac{4c_0}{(m-2)^2(Q-6)}<1.
\end{align*}
Therefore, we prove (\ref{claim})  in the case (\ref{small-1}).

Now, we calculate $I'(r)$.  Differentiating (\ref{I}) with respect to $r$, one gets
\begin{align}\label{I-1}
I'(r)&=I'_{1}(r)+r^4I'_{2}(r)+4r^3I_{2}(r)\nonumber\\[2mm]
&=\int_{ \partial{B_r}} \frac{|\nabla_X u|^2}{|\nabla\rho|}+\int_{\partial{B_r}}\frac{uw}{|\nabla\rho|}
+r^{4}\int_{ \partial{B_r}} \frac{|\nabla_X w|^2}{|\nabla\rho|} \nonumber\\[2mm]
&+r^{4}\int_{\partial{B_r}}\frac{V uw}{|\nabla\rho|}
+4r^{3}\int_{ B_r} |\nabla_X w|^2+4r^{3}\int_{B_r} V uw \nonumber\\[2mm]
& \equiv K_1+K_2+K_3+r^{4}\int_{\partial{B_r}}\frac{V uw}{|\nabla\rho|} +4r^{3}\int_{ B_r} |\nabla_X w|^2+4r^{3}\int_{B_r} V uw.
\end{align}
By using the fact (\ref{Z}) and the   divergence theorem , one has
\begin{align}
K_1&\equiv\int_{ \partial{B_r}} \frac{|\nabla_X u|^2}{|\nabla\rho|}=\frac{1}{r}\int_{ \partial{B_r}} \frac{|\nabla_X u|^2}{|\nabla\rho|}Z\rho\nonumber\\[2mm]
&=\frac{1}{r}\int_{ \partial{B_r}} |\nabla_X u|^2 Z\cdot\frac{\nabla\rho}{|\nabla\rho|}=
\frac{1}{r}\int_{B_r}\mbox{div}(|\nabla_{X}u|^2Z) \nonumber\\[2mm]
&=\frac{Q}{r}\int_{B_r}|\nabla_{X}u|^2+\frac{2}{r}\int_{B_r}\sum_{i}X_iu Z(X_iu).\nonumber
\end{align}
Recalling $[X_i,Z]=X_i$,  using the divergence theorem and equation (\ref{w}), the last term on $K_1$ becomes
\begin{align}
\frac{2}{r}\int_{B_r}\sum_{i}X_iu Z(X_iu)&=\frac{2}{r}\int_{B_r}\sum_{i}X_iu X_iZu-\frac{2}{r}\int_{B_r}\sum_{i}(X_iu)^2\nonumber\\[2mm]
&=-\frac{2}{r}\int_{B_r}w Zu+\frac{2}{r}\int_{\partial{B_r}}Zu\frac{\nabla_{X}u\cdot\nabla_{X}\rho}{|\nabla\rho|} -\frac{2}{r}\int_{B_r}|\nabla_{X}u|^2.\nonumber
\end{align}
Hence, by using the identity (\ref{Z2}), we obtain
\begin{align}
K_1&=\frac{Q-2}{r}\int_{B_r}|\nabla_{X}u|^2-\frac{2}{r}\int_{B_r}w Zu+\frac{2}{r}\int_{\partial{B_r}} Zu \frac{\nabla_{X}u\cdot\nabla_{X}\rho}{|\nabla\rho|}\nonumber\\[2mm]
&=\frac{Q-2}{r}\int_{B_r}|\nabla_{X}u|^2-\frac{2}{r}\int_{B_r}w Zu
+\frac{2}{r^2}\int_{\partial{B_r}} (Zu)^2 \frac{\psi}{|\nabla\rho|}.
\end{align}
The term $K_3$  can be calculated in the same way as $K_1$,
\begin{align}
K_3&\equiv r^{4}\int_{ \partial{B_r}} \frac{ |\nabla_X w|^2}{|\nabla\rho|}\nonumber\\[2mm]
&=\frac{Q-2}{r}r^{4}\int_{B_r}|\nabla_Xw|^2-2r^{3}\int_{B_r}V u Zw
+2r^{2} \int_{\partial{B_r}}(Zw)^2\frac{\psi}{|\nabla\rho|}.
\end{align}
By using the fact (\ref{Z}) and the   divergence theorem again, one has
\begin{align}
K_2&\equiv \int_{\partial{B_r}}\frac{uw}{|\nabla\rho|}=\frac{1}{r}\int_{\partial{B_r}}uw\frac{Z\rho}{|\nabla\rho|}
=\frac{1}{r}\int_{B_r}\mbox{div}(uw  Z)\nonumber\\[2mm]
&=\frac{Q}{r}\int_{B_r}uw+\frac{1}{r}\int_{B_r}w Zu+ \frac{1}{r}\int_{B_r}u Zw.
\end{align}
Putting $K_i$ (i=1,2,3) into (\ref{I-1}) and recalling (\ref{I1}), it holds
\begin{align}\label{sum1}
I'(r)&=\frac{Q-2}{r} I(r)+\frac{2}{r^2}\int_{\partial{B_r}}\frac{(Zu)^2\psi}{|\nabla\rho|}
+2r^2\int_{\partial{B_r}}\frac{(Zw)^2\psi}{|\nabla\rho|}+4r^3\int_{B_r}|\nabla_{X}w|^2\nonumber\\[2mm]
&+\frac{2}{r}\int_{B_r}uw-(Q-6)r^3\int_{B_r}Vuw
-\frac{1}{r}\int_{B_r} Zu w\nonumber\\[2mm]
&+\frac{1}{r}\int_{B_r}Zw u-2r^{3}\int_{B_r}VuZw+r^{4}\int_{\partial B_r}\frac{Vuw}{|\nabla\rho|}
\nonumber\\[2mm]
&\equiv\frac{Q-2}{r} I(r)+\frac{2}{r^2}\int_{\partial{B_r}}\frac{(Zu)^2\psi}{|\nabla\rho|}
+2r^2\int_{\partial{B_r}}\frac{(Zw)^2\psi}{|\nabla\rho|}+4r^3\int_{B_r}|\nabla_{X}w|^2\nonumber\\[2mm]
&+\sum_{i=1}^6R_{i}.
\end{align}
In the following, we estimate $R_i$.
Applying Lemma   \ref {lem3} to $w$ and Lemma \ref{lem4} to $u$, term $R_1$ can be estimated as follows:
\begin{align}\label{mon-1}
R_1&\equiv2r^{-1}\int_{B_r}uw\leq2r^{3}\int_{B_r}\frac{|u||w|}{\rho^4}\nonumber\\[2mm]
&\leq r^{3} \int_{B_r} \frac{u^2\psi}{\rho^6}+ r^{3} \int_{B_r} \frac{w^2}{\rho^2\psi}\nonumber\\[2mm]
&\leq C\left(r^{3} \int_{B_r}|\nabla_{X}w|^2 +r^{2}\int_{\partial{B_r}}\frac{w^2\psi}{|\nabla\rho|}
+ r^{-2}\int_{\partial{B_r}}\frac{u^2\psi}{|\nabla\rho|}\right).
\end{align}
Applying the assumption on $V$ (\ref{V}), using Lemma \ref{lem3} and Lemma  \ref{lem4} again, one has
\begin{align}
R_2&\equiv -(Q-6)r^{3}\int_{B_r}Vuw\leq C  r^{3}\int_{B_r} \frac{|u||w|}{\rho^4}\nonumber\\[2mm]
&\leq Cr^{3} \left(\int_{B_r} \frac{u^2\psi}{\rho^6}+  \int_{B_r} \frac{w^2}{\rho^2\psi}\right)\nonumber\\[2mm]
&\leq C r^{3} \left(\int_{B_r}|\nabla_{X}w|^2 +r^{-1}\int_{\partial{B_r}}\frac{w^2\psi}{|\nabla\rho|}
+ r^{-5}\int_{\partial{B_r}}\frac{u^2\psi}{|\nabla\rho|}\right).
\end{align}
In order to estimate the terms with $Zu$ or $Zw$, we show the following fact:
\begin{equation}\label{fact}
|Zu|\leq \frac{\rho^{\alpha+1}}{|x|^{\alpha}}|\nabla_{X}u|\leq \rho \psi^{-1/2}|\nabla_{X}u|.
\end{equation}
Indeed,
\begin{equation*}
Zu=(x,(\alpha+1)y)\cdot(\partial_{x}u,\partial_{y}u)=\left(x,\frac{(\alpha+1)y}{|x|^{\alpha}}\right)\cdot (\partial_{x}u,|x|^{\alpha}\partial_{y}u).
\end{equation*}
According to (\ref{fact}), and applying Lemma \ref{lem3},  it holds
\begin{align}\label{R3}
R_3&\equiv-r^{-1}\int_{B_r} Zu w
\leq r^{-1}\int_{B_r} \rho \psi^{-1/2}|\nabla _{X}u||w|\leq r\int_{B_r} \frac{|\nabla _{X}u||w|}{\rho \psi^{1/2}}\nonumber\\[2mm]
&\leq C \left(r^{-1}\int_{B_r}  |\nabla _{X}u|^2+  r^3\int_{B_r}\frac{w^2}{\rho^2\psi}\right)\nonumber\\[2mm]
&\leq C \left(r^{-1}\int_{B_r}  |\nabla _{X}u|^2+  r^3\int_{B_r} |\nabla_{X}w|^2+ r^{2} \int_{\partial{B_r}}\frac{w^2\psi}{|\nabla\rho|}\right),
\end{align}
and
\begin{align}\label{R4}
R_4&\equiv r^{-1}\int_{B_r}Zw u \leq r\int_{B_r}  \frac{|u| |\nabla _{X}w|}{\rho\psi^{1/2}}\nonumber\\[2mm]
&\leq C\left(  r^3\int_{B_r} |\nabla _{X}w|^2+r^{-1} \int_{B_r} \frac{|u|^2}{\rho^2\psi}\right)\nonumber\\[2mm]
&\leq C \left(r^3\int_{B_r} |\nabla _{X}w|^2+r^{-1} \int_{B_r} |\nabla_{X}u|^2+ r^{-2} \int_{\partial{B_r}}\frac{u^2\psi}{|\nabla\rho|}\right).
\end{align}
Using (\ref{fact}) again, and applying Lemma \ref{lem5}, one has
\begin{align}\label{R5}
R_5&\equiv-2r^{3}\int_{B_r}VuZw  \leq Cr^{3}\int_{B_r}\frac{|u||Zw|}{\rho^4}
\leq Cr^{3}\int_{B_r}\frac{|u||\nabla_{X}w|}{\rho^3\psi^{1/2}}\nonumber\\[2mm]
&\leq  Cr^{3}\left(\int_{B_r}\frac{u^2}{\rho^6\psi} + \int_{B_r}|\nabla_{X}w|^2\right) \nonumber\\[2mm]
&\leq Cr^{3}\left(\int_{B_r}|\nabla_Xw|^2+r^{-4}\int_{B_r}|\nabla_{X} u|^2+r^{-1}\int_{\partial B_r} \frac{w^2\psi}{|\nabla\rho|}+r^{-5}\int_{\partial B_r}\frac{u^2\psi}{|\nabla\rho|}\right).
\end{align}
Under assumption (\ref{V}), we estimate the boundary term $R_6$ as follows,
\begin{align}\label{R6}
R_6&\equiv r^{4}\int_{\partial{B_r}}\frac{V wu}{|\nabla\rho|}\leq c_0\int_{\partial{B_r}}\frac{|w||u|}{|\nabla\rho|}\nonumber\\[2mm]
&=c_0 r^{-1}\int_{\partial{B_r}}\frac{|w||u|Z\cdot\nabla\rho}{|\nabla\rho|}
=c_0r^{-1}\int_{{B_r}}\mbox{div}\left(Z|w||u|\right)\nonumber\\[2mm]
&= c_0r^{-1}\left(\int_{{B_r}}\mbox{div}Z|w||u| +\int_{{B_r}}|u|Z|w|+\int_{{B_r}}|w|Z|u| \right)\nonumber\\[2mm]
&\equiv R_{61}+ R_{62}+ R_{63}.
\end{align}
Recalling the fact (\ref{Z0}), we have
\begin{align}
 R_{61}&=c_0Qr^{-1}\int_{{B_r}}|w||u|
\leq C r^{3} \int_{{B_r}}\frac{|w||u|}{\rho^4} \nonumber\\[2mm]
 &\leq  C r^{3}\left(\int_{B_r} \frac{u^2\psi}{\rho^6}+   \int_{B_r} \frac{w^2}{\rho^2\psi}\right)\nonumber\\[2mm]
&\leq   Cr^{3}\left(\int_{B_r}|\nabla_{X}w|^2 + r^{-1}\int_{\partial{B_r}}\frac{w^2\psi}{|\nabla\rho|}
+ r^{-5}\int_{\partial{B_r}}\frac{u^2\psi}{|\nabla\rho|}\right).
 \end{align}
By (\ref{fact}) again, similarly as for (\ref{R3}), one has
\begin{align}
 R_{62}&\equiv C r^{-1}\int_{{B_r}}|u|Z|w| \leq C r^{-1}\int_{B_r}|Zw||u|\nonumber\\[2mm]
 &\leq C \left( r^{3}\int_{B_r}  |\nabla _{X}w|^2+ r^{-1} \int_{B_r} |\nabla_{X}u|^2+ r^{-2} \int_{\partial{B_r}}\frac{u^2\psi}{|\nabla\rho|}\right),
 \end{align}
 and
 \begin{align}\label{mon-2}
 R_{63}&\equiv C r^{-1}\int_{{B_r}}|w|Z|u|\leq C r^{-1}\int_{B_r}|Zu||w|\nonumber\\[2mm]
 &\leq C \left( r^{-1}\int_{B_r}  |\nabla _{X}u|^2+ r^{3} \int_{B_r} |\nabla_{X}w|^2+ r^{2} \int_{\partial{B_r}}\frac{w^2\psi}{|\nabla\rho|}\right).
 \end{align}
Putting (\ref{mon-1})-(\ref{mon-1}) into (\ref{sum1}), and   according to the claim (\ref{claim}) and (\ref{h<i}), we arrive at
\begin{align*}
\left|I'(r)-\left( \frac{Q-2}{r}I(r)+ \frac{2}{r^2}\int_{\partial{B_r}}\frac{(Zu)^2\psi}{|\nabla\rho|}
+2r^2\int_{\partial{B_r}}\frac{(Zw)^2\psi}{|\nabla\rho|}\right)\right|\leq C\frac{I(r)}{r}.
\end{align*}
 This finishes the proof of Lemma \ref{lemma2}.
\end{proof}

Based on  Lemmas \ref{lemma1}, \ref{lemma2}, the monotonicity of the frequency function can be be easily established.

{\bf Proof of Theorem \ref{thm4}}.\
Using (\ref{H'(r)}) and (\ref{I'(r)}), we finally obtain
\begin{align*}
\frac{N'(r)}{N(r)}&=\frac{1}{r}+\frac{I'(r)}{I(r)}-\frac{H'(r)}{H(r)}\\
&\geq \frac{2\left(\int_{\partial B_r}u_n^2+r^4\int_{\partial B_r}w_n^2\right)}{\int_{\partial B_r}u u_n+r^4\int_{\partial B_r}w w_n}
-\frac{2\left(\int_{\partial B_r}u u_n+r^4\int_{\partial B_r}w w_n\right)}{\int_{\partial B_r}u^2+r^{4}\int_{\partial B_r}w^2}-\frac{C}{r}-\frac{4}{r}\\
&\geq-\frac{\beta}{r},
\end{align*}
where we have used the Schwarz's inequality. This finishes the proof.\qed

\section{Vanishing order and strong unique continuation property}

To prove Theorem \ref{thm0}, we need to check that if the vanishing order of  a solution $u$ to $Lu=0$ is infinite, then  $\Delta_{X}u$ also vanishes to infinite order. Precisely,

\begin{theorem}\label{thm3}
Let $u\in M^{2,2}_{loc}(\Omega)$ be a solution of (\ref{L}), where  $V$  satisfies the growth assumption (\ref{V}). If $u$ vanishes to infinite order  at the origin in the sense of (\ref{exp}), then $\Delta_{X}u$ also vanishes to infinite order at the origin, that is
\begin{align}\label{exp-2}
\int_{B_r}(\Delta_{X}u)^2\psi=O\left(\exp(-\widetilde{B}r^{-\gamma})\right),
\end{align}
as $r\rightarrow0$ for some constants $\widetilde{B},\gamma>0$.
\end{theorem}

\begin{proof}
Let $\eta(t)\in C_0^{\infty}(\mathbb{R})$ be a cut-off function such that
\begin{equation*}
\eta(t)=\left\{
\begin{array}{lll}1\quad 0\leq t\leq r,\\
0\quad t\geq 2r.
\end{array}\right.
\end{equation*}
Moreover, $0\leq\eta\leq1$, and
\begin{align*}
| \eta'(t)|\leq \frac{C}{r},\quad \mbox{and}\hspace{2mm}|\eta''(t)|\leq \frac{C}{r^2}, \quad \mbox{for} \hspace{2mm} r\leq t\leq 2r.
\end{align*}
Choosing cut-off function $\eta(\rho(z))$,  from the properties of the gauge norm  (\ref{psi}), (\ref{psi2}), it holds
\begin{align}\label{test}
 &\left|X \eta(\rho(z))\right|=\left|\eta' X \rho\right|\leq \frac{C}{r}\psi^{1/2}\nonumber\\[2mm]
 &\left|X_iX_j\eta(\rho(z))\right|=\left|\eta'' X_i\rho X_j\rho+\eta' X_iX_j\rho\right|\leq \frac{C}{r^2},\quad z\in B_{2r}\backslash B_{r}.
\end{align}
Since $u$ is a solution of (\ref{L}), it holds
\begin{align}\label{L1}
\int_{\Omega}\Delta_{X}^2 u\phi=\int_{\Omega}Vu\phi,
\end{align}
for any $\phi\in M^{2,2}_0(\Omega)$.
Taking the test function $\phi=u\eta^4$ in (\ref{L1}) and integrating by parts twice, we deduce that
\begin{align*}
&\int_{B_{2r}}Vu^2\eta^4=\int_{B_{2r}}\Delta_{X}^2u (u\eta^4)
=\int_{B_{2r}}\Delta_{X} u \Delta_{X}(u\eta^4),
\end{align*}
which yields
\begin{align}\label{inest}
\int_{B_{2r}}(\Delta_{X} u)^2\eta^4&=\int_{B_{2r}}Vu^2\eta^4-4\int_{B_{2r}}\eta^3u\Delta_{X} u \Delta_{X}\eta\nonumber\\[2mm]
&-12\int_{B_{2r}}\eta^2u\Delta_{X} u|\nabla_{X} \eta|^2-8\int_{B_{2r}}\eta^3\Delta_{X} u\nabla_{X} u\cdot\nabla_{X} \eta\nonumber\\[2mm]
&\equiv \int_{B_{2r}}Vu^2\eta^4+\sum_{i=1}^{3}J_i.
\end{align}
Using the Cauchy inequality and the properties of the test function (\ref{test}), it holds
\begin{align}\label{inestJ1}
J_1&\leq C\int_{B_{2r}}\eta^3\left|u\Delta_{X} u \Delta_{X}\eta\right| \nonumber\\[2mm]
&\leq \frac{1}{8} \int_{B_{2r}}(\Delta_{X} u)^2\eta^4+C \int_{B_{2r}}u^2 \eta^2(\Delta_{X}\eta)^2,
\end{align}
\begin{align}\label{inestJ2}
J_2&\leq C\int_{B_{2r}}\eta^2\left|u\Delta_{X} u\right| |\nabla_{X} \eta|^2\nonumber\\[2mm]
&\leq \frac{1}{8} \int_{B_{2r}}(\Delta_{X} u)^2\eta^4+C \int_{B_{2r}}u^2 |\nabla_{X} \eta|^4,
\end{align}
and
\begin{align}\label{inestJ3}
J_3&\leq C\int_{B_{2r}}\eta^3\Delta_{X} u\nabla_{X} u\cdot\nabla_{X} \eta \nonumber\\[2mm]
&\leq \frac{1}{8} \int_{B_{2r}}(\Delta_{X} u)^2\eta^4+C \int_{B_{2r}} \eta^2|\nabla_{X}\eta|^2|\nabla_{X}u|^2.
\end{align}
Next, we estimate the last term on the right-hand side  of (\ref{inestJ3}).
Integrating by parts  deduces that
\begin{align*}
&\int_{B_{2r}}\eta^2|\nabla_{X} u|^2|\nabla_{X}\eta|^2=-\int_{B_{2r}}\mbox{div}_{X}\left(\nabla_{X}u|\nabla_{X}\eta|^2\eta^2\right)u\\[2mm]
&=-\int_{B_{2r}}\Delta_{X}u|\nabla_{X}\eta|^2\eta^2u-2\int_{B_{2r}}X_kuX_{k}(X_j\eta)X_{j}\eta \eta^2u-2\int_{B_{2r}}\eta u|\nabla_{X}\eta|^2\nabla_{X}u\cdot\nabla_{X}\eta\\[2mm]
&\leq\frac{1}{16}\int_{B_{2r}}|\Delta_{X}u|^2\eta^4+\frac{1}{2}\int_{B_{2r}}|\nabla_{X}u|^2|\nabla_{X}\eta|^2\eta^2+
C\left(\int_{B_{2r}}u^2|\nabla_{X}\eta|^4+\int_{B_{2r}}|\nabla^2_{X}\eta|^2\eta^2u^2\right),
\end{align*}
which gives
\begin{align}\label{key}
\int_{B_{2r}}\eta^2|\nabla_{X} u|^2|\nabla_{X}\eta|^2
\leq \frac{1}{8}\int_{B_{2r}}|\Delta_{X}u|^2\eta^4+C\left(\int_{B_{2r}}u^2|\nabla_{X}\eta|^4+\int_{B_{2r}}u^2|\nabla^2_{X}\eta|^2\eta^2\right).
\end{align}
Putting (\ref{inestJ1})-(\ref{key}) into (\ref{inest}), we see that
\begin{align}\label{vanish}
\int_{B_{r}}(\Delta_{X} u)^2\leq 2\int_{B_{2r}}Vu^2\eta^4+Cr^{-4}\int_{B_{2r}}u^2.
\end{align}
Since $u$  vanishes to infinite order  at the origin in the sense of (\ref{exp}), the assumption on $V$ (\ref{V}) and Proposition \ref{prop} infer that
\begin{align*}
\int_{B_{2r}}Vu^2\eta^4&\leq c_0\int_{B_{2r}}\frac{u^2}{\rho^4}=c_0\sum_{j=0}^{\infty}\int_{2^{-j}r\leq\rho(z)\leq 2^{-(j-1)}r}\frac{u^2}{\rho^4} \\[2mm]
&\leq c_0\sum_{j=0}^{\infty}2^{4j}r^{-4} \int_{\rho(z)\leq 2^{-(j-1)}r}u^2 \\[2mm]
&\leq C \sum_{j=0}^{\infty}2^{4j}r^{-4} e^{-B2^{(j-1)\gamma}r^{-\gamma}} \\[2mm]
&\leq Ce^{-\widetilde{B}(2r)^{-\gamma}}.
\end{align*}
Then, both terms on the right side of (\ref{vanish}) vanish to infinite order  and the theorem
follows.
\end{proof}

With Theorem \ref{thm1}, Theorem \ref{thm3} at hand,  the strong unique continuation property can be obtained directly.  The argument is
standard (see e.g. \cite{gl1}), we include it for the sake of completeness.

{\bf Proof of Theorem \ref{thm0}}.
Making use of Theorem \ref{thm3} and Proposition \ref{prop},  the assumption (\ref{exp}) imply that
\begin{align*}
\int_{B_r}\left(u^2+(\Delta_{X}u)^2\right)\psi\leq C \exp(-Br^{-\gamma}), \quad \mbox{for some constants}\quad B,\gamma>0.
\end{align*}
Now, for fixed $R$,  after $k$ times iterations of (\ref{doubling}), we infer
\begin{align*}
\int_{B_{R}}\left(u^2 + (\Delta_{X}u)^2\right)\psi&\leq
\Big(C\exp\left(A R^{-\gamma}\right)\Big)^{k}\int_{B_{2^{-k}R}}\left(u^2 + (\Delta_{X}u)^2\right)\psi\\[2mm]
&\leq\Big(C\exp\left(A R^{-\gamma}\right)\Big)^{k}\exp\left(-B(2^{-k}R)^{-\gamma}\right)\\[2mm]
&\leq \exp\left((k\ln C+kA-B2^{k\gamma})R^{-\gamma}\right)\rightarrow 0\quad \mbox{as} \hspace{1mm} k\rightarrow \infty.
\end{align*}
Then $u\equiv 0$ in $B_{R}$. Moreover, for some point $(0,y_0)\in B_{R}$, we have $u\equiv0$ in the neighbourhood of $(0,y_0)$, so $u$ vanishes to infinite order at $(0,y_0)$.
Since the operator $\Delta^2_{X}$ is translation invariant in $y$, we can  repeat the previous process in the gauge ball $B_{r}(0,y_0)$ to show that $u\equiv 0$ in $B_{R}(0,y_0)$. On the other hand, outside of $B_{R} (0,y_0)$, the equation (\ref{L}) can be seen as bi-Laplace with  bounded potential $V$, therefore we can apply the results in \cite{liu2022} to conclude that $u\equiv0$ in $\Omega$.
\qed

\section*{Funding}
The research of the first author was supported by the National Natural Science Foundation of China (No.12071219). The research of the second author was supported by the National Natural Science Foundation of China (No.11971229).

\end{document}